\documentclass[a4paper,10pt]{amsart}

\usepackage[left=1.33in,right=1.33in,top=1.1in,bottom=1.1in,includehead,includefoot]{geometry}
\usepackage[T1]{fontenc}
\usepackage{lmodern,microtype}
\usepackage{amsthm,amsmath,amssymb,colonequals}
\usepackage{graphicx}
\usepackage{enumitem}
\usepackage[pdftitle={Drawing graphs using a small number of obstacles},
            pdfauthor={Martin Balko, Josef Cibulka, Pavel Valtr}]{hyperref}

\newtheorem{theorem}{Theorem}
\newtheorem{lemma}[theorem]{Lemma}
\newtheorem{corollary}[theorem]{Corollary}

\newtheorem{claim}[theorem]{Claim}
\newtheorem*{claim*}{Claim}
\newtheorem{remark}{Remark}

\makeatletter
\let\old@setaddresses\@setaddresses
\def\@setaddresses{\bgroup\parindent 0pt\let\scshape\relax\old@setaddresses\egroup}
\makeatother

\DeclareMathOperator{\obs}{obs}
\DeclareMathOperator{\obsc}{obs_c}

\title{Drawing graphs using a small number of obstacles}

\author{Martin Balko}
\author{Josef Cibulka}
\author{Pavel Valtr}

\address[Martin Balko, Josef Cibulka, Pavel Valtr]{Department of Applied Mathematics and Institute for Theoretical Computer Science, Faculty of Mathematics and Physics, Charles University, Prague, Czech Republic}
\email{balko@kam.mff.cuni.cz, cibulka@kam.mff.cuni.cz}
\address[Martin Balko]{Alfr\'{e}d R\'{e}nyi Institute of Mathematics, Hungarian Academy of Sciences, Budapest, Hungary}

\thanks{The first and the third author acknowledge the support of the project CE-ITI (GA\v{C}R P202/12/G061) of the Czech Science Foundation, ERC Advanced Research Grant no 267165 (DISCONV), and the grant GAUK 1262213 of the Grant Agency of Charles University.
The first author was also supported by the grant SVV--2016--260332.
Part of the research was conducted during the workshop Homonolo 2014 supported by the European Science Foundation as a part of the EuroGIGA collaborative research program (Graphs in Geometry and Algorithms).
An extended abstract of this paper appeared in Proceedings of the 23rd Symposium on Graph Drawing, 2015.}

\begin{document}

\begin{abstract}
An \emph{obstacle representation} of a graph $G$ is a set of points in the plane representing the vertices of $G$, together with a set of polygonal obstacles such that two vertices of~$G$ are connected by an edge in $G$ if and only if the line segment between the corresponding points avoids all the obstacles.
The \emph{obstacle number $\obs(G)$ of $G$} is the minimum number of obstacles in an obstacle representation of $G$.

We provide the first non-trivial general upper bound on the obstacle number of graphs by showing that every $n$-vertex graph $G$ satisfies $\obs(G) \leq n\lceil\log{n}\rceil-n+1$.
This refutes a conjecture of Mukkamala, Pach, and P{\'a}lv{\"o}lgyi.
For $n$-vertex graphs with bounded chromatic number, we improve this bound to $O(n)$.
Both bounds apply even when the obstacles are required to be convex.

We also prove a lower bound $2^{\Omega(hn)}$ on the number of $n$-vertex graphs with obstacle number at most $h$ for $h<n$ and a lower bound $\Omega(n^{4/3}M^{2/3})$ for the complexity of a collection of $M \geq \Omega(n\log^{3/2}{n})$ faces in an arrangement of line segments with $n$ endpoints.
The latter bound is tight up to a multiplicative constant.
\end{abstract}

\maketitle

\section{Introduction}

In a \emph{geometric drawing} of a graph $G$, the vertices of $G$ are represented by distinct points in the plane and each edge $e$ of $G$ is represented by the line segment between the pair of points that represent the vertices of $e$.
As usual, we identify the vertices and their images, as well as the edges and the line segments representing them.

Let $P$ be a finite set of points in the plane in \emph{general position}, 
that is, there are no three collinear points in $P$.
The \emph{complete geometric graph} $K_P$ is the geometric drawing of the complete graph $K_{|P|}$ with vertices represented by the points of $P$.

An \emph{obstacle} is a polygon in the plane.
An \emph{obstacle representation} of a graph $G$ is a geometric drawing $D$ of $G$ together with a set $\mathcal{O}$ of obstacles such that two vertices of $G$ are connected by an edge $e$ if and only if the line segment representing $e$ in $D$ is disjoint from all obstacles in $\mathcal{O}$.
The \emph{obstacle number} $\obs(G)$ of $G$ is the minimum number of obstacles in an obstacle representation of $G$.
The \emph{convex obstacle number} $\obsc(G)$ of a graph $G$ is the minimum number of obstacles in an obstacle representation of $G$ in which all the obstacles are required to be convex.
Clearly, we have $\obs(G) \leq \obsc(G)$ for every graph $G$.
For a positive integer $n$, let $\obs(n)$ be the maximum obstacle number of a graph on $n$ vertices.

In this paper, we provide the first nontrivial upper bound on $\obs(n)$ (Theorem~\ref{thm-GeneralUpperBound}).
We also show a lower bound for the number of graphs with small obstacle number (Theorem~\ref{thm-count-lb}) and a matching lower bound for the complexity of a collection of faces in an arrangement of line segments that share endpoints (Theorem~\ref{thm-arrangements}).
All proofs of our results are based on so-called \emph{$\varepsilon$-dilated bipartite drawings of $K_{m,n}$}, which we introduce in Section~\ref{sec:bipartiteDrawings}.

In the following, we make no serious effort to optimize the constants.
All logarithms in this paper are base 2.

\subsection{Bounding the obstacle number}
\label{subsec:obstacleNumbers}

The obstacle number of a graph was introduced by Alpert, Koch, and Laison~\cite{Alpert+10} who showed, among several other results, that for every positive integer $h$ there is a graph $G$ with $\obs(G) \geq h$.
Using extremal graph theoretic tools, Pach and Sar{\i}{\"o}z~\cite{PachSarioz} proved that the number of labeled $n$-vertex graphs with obstacle number at most $h$ is at most $2^{o(n^2)}$ for every fixed integer $h$.
This implies that there are bipartite graphs with arbitrarily large obstacle number.

Mukkamala, Pach, and Sar{\i}{\"o}z~\cite{Mukkamala+10} established more precise bounds by showing that the number of labeled $n$-vertex graphs with obstacle number at most $h$ is at most $2^{O(hn\log^2{n})}$ for every fixed positive integer $h$.
It follows that $\obs(n) \geq \Omega(n/\log^2{n})$. 
Later, Mukkamala, Pach, and P{\'a}lv{\"o}lgyi~\cite{Mukkamala+11} improved the lower bound to $\obs(n) \geq \Omega(n/\log{n})$.
Currently, the strongest lower bound on the obstacle number is due to Dujmovi{\'c} and Morin~\cite{DujmovicMorin} who showed $\obs(n) \geq \Omega(n/(\log{\log{n}})^2)$.

Surprisingly, not much has been done for the general upper bound on the obstacle number.
We are only aware of the trivial bound $\obs(G) \leq \binom{n}{2}$ for every graph $G$ on $n$ vertices.
This follows easily, as we can consider the complete geometric graph $K_P$ for some point set $P$ of size $n$ and place a small obstacle $O_e$ on every \emph{non-edge} $e$ of $G$ such that $O_e$ intersects only $e$ in $K_P$.
A non-edge of a graph $G=(V,E)$ is an element of $\binom{V}{2}\setminus E$.

Concerning special graph classes, Fulek, Saeedi, and Sar{\i}{\"o}z~\cite{Fulek+13} showed that the convex obstacle number satisfies $\obs_c(G) \leq 5$ for every outerplanar graph $G$ and $\obs_c(H) \leq 4$ for every bipartite permutation graph $H$.

Chaplick et al.~\cite{Chaplick+16} showed that $\obs(n) \le 1$ whenever $n \le 7$ and that $\obs(8) \ge 2$.
They also found an $11$-vertex graph $G$ with $\obs(G) = 1$ such that in every representation of $G$ with a single obstacle, the obstacle lies in one of the bounded cells of the drawing of $G$.
Berman et al.~\cite{Berman+16} constructed a $10$-vertex planar graph $G$ with $\obs(G) \ge 2$.

Alpert, Koch, and Laison~\cite{Alpert+10} asked whether the obstacle number of every graph on $n$ vertices can be bounded from above by a linear function of $n$.
We show that this is true for bipartite graphs and split graphs, even for the convex obstacle number.
A \emph{split graph}  is a graph for which the vertex set can be partitioned into an  independent subset and a clique.

\begin{theorem}
\label{thm-BipartiteUpperBound}
If $G=(V,E)$ is a bipartite graph or a split graph, then we have
\[
\obs_c(G),\obs_c(\overline{G}) \leq |V|-1,
\]
where $\overline{G}=(V,E\setminus\binom{V}{2})$ denotes the complement of $G$.
\end{theorem}

On the other hand, a modification of the proof of the lower bound by Mukkamala, Pach, and P{\'a}lv{\"o}lgyi~\cite{Mukkamala+11} implies that there are bipartite graphs $G$ on $n$ vertices with $\obs(G) \geq \Omega(n/\log{n})$ for every positive integer $n$.

In contrast to the above question of Alpert, Koch, and Laison on the existence of a linear upper bound, Mukkamala, Pach, and P{\'a}lv{\"o}lgyi~\cite{Mukkamala+11} conjectured that the maximum obstacle number of $n$-vertex graphs is around $n^2$.
We refute this conjecture by showing the first non-trivial general upper bound on the obstacle number of graphs.
In fact, we prove a stronger result that provides a general upper bound for the convex obstacle number.

\begin{theorem}
\label{thm-GeneralUpperBound}
For every positive integer $n$ and every graph $G$ on $n$ vertices, the convex obstacle number of $G$ 
satisfies
\[
\obs_c(G) \leq  n\lceil \log{n} \rceil - n + 1.
\]
\end{theorem}

The question whether the upper bound on $\obs(n)$ can be improved to $O(n)$ remains open.
We can, however, prove $\obs_c(G) \leq O(n)$ provided that the chromatic number (or even the \emph{subchromatic number}) of $G$ is bounded from above by a constant.

Let $G=(V,E)$ be a graph and let $c \colon V \to \{1,\dots,k\}$ be a function that assigns a \emph{color} $c(v)$ to every vertex $v$ of $G$.
We call $c$ a \emph{subcoloring of $G$} if for every $a \in \{1,\dots,k\}$ the color class $f^{-1}(a)$ induces a disjoint union of cliques in $G$.
The \emph{subchromatic number of $G$}, denoted by $\chi_s(G)$, is the least number of colors needed in any subcoloring of $G$.

We prove the following result  that asymptotically implies Theorems~\ref{thm-BipartiteUpperBound} and~\ref{thm-GeneralUpperBound}.

\begin{theorem}
\label{thm-chromaticUpperBound}
For every positive integer $n$ and every graph $G$ on $n$ vertices, the convex obstacle number of $G$ 
satisfies
\[
\obs_c(G) \leq (n-1)(\lceil\log{\chi_s(G)}\rceil + 1),
\]
\end{theorem}

Note that we have $\chi_{s}(G) \leq \min\{\chi(G),\chi(\overline{G})\}$, thus Theorem~\ref{thm-chromaticUpperBound} gives the  bound $\obs_c(G) \leq O(n)$ if the chromatic number of $G$ or the chromatic number of its complement $\overline{G}$ is bounded from above by a constant.
Theorem~\ref{thm-chromaticUpperBound} immediately implies an analogous statement for \emph{cochromatic number of $G$}, which is the least number of colors that we need to color $G$ so that each color class induces either a clique or an independent set in $G$.
The cochromatic number of a graph $G$ was first defined by Lesniak and Straight~\cite{LesStr77} and it is at least $\chi_s(G)$ for every graph $G$.

\subsection{Number of graphs with small obstacle number}
\label{subsec:numberOfGraphs}

For positive integers $h$ and $n$, let $g(h,n)$ be the number of labeled $n$-vertex graphs with obstacle number at most $h$.
The lower bounds on the obstacle number by Mukkamala, Pach, and P{\'a}lv{\"o}lgyi~\cite{Mukkamala+11} and by Dujmovi{\'c} and Morin~\cite{DujmovicMorin} are both based on the upper bound $g(h,n) \leq 2^{O(hn \log^2 n)}$.
In fact, any improvement on the upper bound for $g(h,n)$ will translate into an improved lower bound on the obstacle number~\cite{DujmovicMorin}.
Dujmovi{\'c} and Morin~\cite{DujmovicMorin} conjectured $g(h,n) \leq 2^{f(n) \cdot o(h)}$ where $f(n) \leq O(n\log^2{n})$.
We show the following lower bound on $g(h,n)$.

\begin{theorem}
\label{thm-count-lb}
For every pair of integers $n$ and $h$ satisfying $0<h < n$, we have
\[
g(h,n) \geq 2^{\Omega(hn)}.
\]
\end{theorem}

This lower bound on $g(h,n)$ is not tight in general.
For example, it is not tight for $h \le o(\log{n})$, as we have $g(h,n) \geq g(1,n)$ and the following simple argument gives $g(1,n) \geq 2^{\Omega(n\log{n})}$.
For a given bijection $f \colon \{1,\dots,\lfloor n/2 \rfloor\} \to \{\lceil n/2 \rceil+1,\dots,n\}$, let $G_f$ be the graph on $\{1,\dots,n\}$ that is obtained from $K_n$ by removing the edges $\{i,f(i)\}$ for every $i \in \{1,\dots,\lfloor n/2 \rfloor\}\}$.
We choose a geometric drawing of $K_n$ where each vertex $i$ is represented by a point $p_i$ such that all line segments $p_ip_{f(i)}$ with $i \in \{1,\dots,\lfloor n/2 \rfloor\}\}$ meet in a common point that is not contained in any other edge of the drawing of $K_n$.
Then we place a single one-point obstacle on this common point and obtain an obstacle representation of $G_f$.
Since the number of the graphs $G_f$ is $(\lfloor n/2 \rfloor)! \geq 2^{\Omega(n\log{n})}$, the super-exponential lower bound on $g(1,h)$ follows.

\subsection{Complexity of faces in arrangements of line segments}
\label{subsec:arrangements}

An \emph{arrangement $\mathcal{A}$ of line segments} is a finite collection of line segments in the plane.
The line segments of $\mathcal{A}$ partition the plane into \emph{vertices}, \emph{edges}, and \emph{cells}.
A vertex is a common point of two or more line segments. 
Removing the vertices from the line segments creates a collection of subsegments which are called edges. 
The cells are the connected components of the complement of the line segments.
A \emph{face} of $\mathcal{A}$ is a closure of a cell.

Note that every geometric drawing of a graph is an arrangement of line segments and vice versa.
The edges of the graph correspond to the line segments of the arrangement and
the vertices of the graph correspond to the endpoints of the line segments.

A line segment $s$ of $\mathcal{A}$ is \emph{incident} to a face $F$ of $\mathcal{A}$ if $s$ and $F$ share an edge of $\mathcal{A}$.
The \emph{complexity of a face} $F$ is the number of the line segments of $\mathcal{A}$ that are incident to $F$.
If $\mathcal{F}$ is a set of faces of $\mathcal{A}$, then the \emph{complexity of $\mathcal{F}$} is the sum of the complexities of $F$ taken over all $F \in \mathcal{F}$.

An \emph{arrangement of lines} is a finite collection of lines in the plane with faces and their complexity defined analogously.

Edelsbrunner and Welzl~\cite{EdelsbrunnerWelzl} constructed an arrangement of $m$ lines having a set of $M$ faces with complexity $\Omega(m^{2/3}M^{2/3} + m)$ for every $m$ and $M \le \binom{m}{2}+1$.
Wiernik and Sharir~\cite{WiernikSharir} constructed an arrangement of $m$ line segments with a single face of complexity $\Omega(m \alpha(m))$.
These two constructions can be combined to provide the lower bound $\Omega(m^{2/3}M^{2/3} + m \alpha(m))$ for the complexity of $M$ faces in an arrangement of $m$ line segments, where $M \le \binom{m}{2}+1$.
The best upper bound for the complexity of $M$ faces in an arrangement of $m$ line segments is $O(m^{2/3}M^{2/3} + m \alpha(m) + m \log{M})$ by Aronov et al.~\cite{Aronov+92}.

Arkin et al.~\cite{Arkin+95} studied arrangements whose line segments share endpoints.
That is, they considered the maximum complexity of a face when we bound the number of endpoints of the line segments instead of the number of the line segments.
They showed that the complexity of a single face in an arrangement of line segments with $n$ endpoints is at most $O(n \log{n})$.
An $\Omega(n \log{n})$ lower bound was then proved by Matou\v{s}ek and Valtr~\cite{MatousekValtr}.

Arkin et al.~\cite{Arkin+95} asked what is the maximum complexity of a set of $M$ faces in an arrangement of line segments with $n$ endpoints.

Since every arrangement of line segments with $n$ endpoints contains at most $\binom{n}{2}$ line segments, the upper bound $O(n^{4/3}M^{2/3}+n^2\alpha(n)+n^2\log{M})$ can be deduced from the bound by Aronov et al.~\cite{Aronov+92}.
Together with the upper bound $O(nM\log{n})$, which follows from the bound by Arkin et al.~\cite{Arkin+95} on the complexity of a single face, we obtain an upper bound $O(\min\{nM\log{n},n^{4/3}M^{2/3}+n^2\log{M}\})$ on the complexity of $M$ faces in an arrangement of line segments with $n$ endpoints.

We give the following lower bound.

\begin{theorem}
\label{thm-arrangements}
For every sufficiently large integer $n$, there is an arrangement $\mathcal{A}$ of line segments with $n$ endpoints such that for every $M$ satisfying  $1\leq M \leq n^4$ there is a set of at most $M$ faces of $\mathcal{A}$ with complexity $\Omega(\min\{nM,n^{4/3}M^{2/3}\})$.
\end{theorem}

Whenever $M \geq n \log^{3/2}{n}$, this lower bound matches the upper bound up to a multiplicative constant.
For $M<n\log^{3/2}{n}$, the lower bound differs from the upper bound by at most an $O(\log{n})$ multiplicative factor; see also Table~\ref{tab:arrgBounds}.

\begin{table}
\centering
\begin{tabular}{l|l|l|l}
$M \le O(1)$ & $\Omega(1) \le M \le O(n)$ & $\Omega(n) \le M \le O(n \log^{\frac{3}{2}} n)$ & $\Omega(n \log^{\frac{3}{2}} n) \le M \le O(n^4)$ \\
\hline
$O(n \log n)$~\cite{Arkin+95} & $O(nM \log n)$~\cite{Aronov+92} & $O(n^2 \log M)$~\cite{Aronov+92} & $O(n^{4/3} M^{2/3})$~\cite{Aronov+92} \\
\hline
$\Omega(n \log n)$~\cite{MatousekValtr} & $\Omega(nM)$ & $\Omega(n^{4/3} M^{2/3})$ & $\Omega(n^{4/3} M^{2/3})$ 
\end{tabular}
\caption{A summary of the best known upper and lower bounds on the complexity of $M$ faces in an arrangement of segments with $n$ endpoints.}
\label{tab:arrgBounds}
\end{table}

\section{Dilated bipartite drawings}
\label{sec:bipartiteDrawings}

For a point $p\in \mathbb{R}^2$, let $x(p)$ and $y(p)$ denote the $x$- and the $y$-coordinate of $p$, respectively.
An \emph{intersection point} in a geometric drawing $D$ of a graph $G$ is a common point of two edges of $G$ that share no vertex.

Let $m$ and $n$ be positive integers.
We say that a geometric drawing of $K_{m,n}$  is \emph{bipartite} if the vertices of the same color class of $K_{m,n}$ lie on a common vertical line and not all vertices of $K_{m,n}$ lie on the same vertical line.
For the rest of this section, we let $D$ be a bipartite drawing of $K_{m,n}$ and use $P\colonequals\{p_1,\ldots,p_m\}$ and $Q\colonequals\{q_1,\ldots,q_n\}$ with $y(p_1)<\cdots<y(p_m)$ and $y(q_1)<\cdots<y(q_n)$ to denote the point sets representing the color classes of $K_{m,n}$ in~$D$.
We let $\ell_P$ and $\ell_Q$ be the vertical lines that contain the points of $P$ and $Q$, respectively.
The \emph{width} $w$ of $D$ is $|x(q_1)-x(p_1)|$.
In the following, we assume that $\ell_P$ is to the left of $\ell_Q$ and that $p_1=(0,0)$ and $q_1=(w,0)$.
We set $d_i\colonequals y(p_{i+1})-y(p_i)$ for $i=1,\ldots,m-1$ and $h_j\colonequals y(q_{j+1})-y(q_j)$ for $j=1,\ldots,n-1$.
We call $d_1$ the \emph{left step} of $D$ and $h_1$ the \emph{right step} of $D$.
An example can be found in part~(a) of  Figure~\ref{fig-bipDraw}.

\begin{figure}[ht]
\centering
\includegraphics{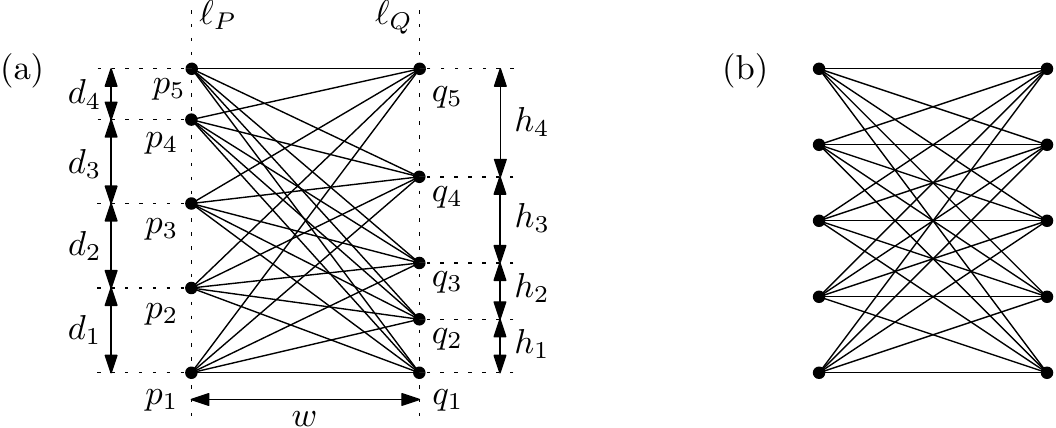}
\caption{(a) A bipartite drawing of $K_{5,5}$ that is not regular. (b) A regular drawing of $K_{5,5}$.
}
\label{fig-bipDraw}
\end{figure}

We say that $D$ is \emph{regular} if we have $d_1=\cdots=d_{m-1}$ and $h_1=\cdots=h_{n-1}$; see part~(b) of  Figure~\ref{fig-bipDraw} for an example.
Note that every regular drawing of $K_{m,n}$ is uniquely determined by its width, left step, and right step.
A \emph{regularization} of a (possibly non-regular) bipartite drawing~$D$ is the regular bipartite drawing of $K_{m,n}$ with the vertices $\pi(p_i) \colonequals (0,(i-1)d_1)$ and $\pi(q_j) \colonequals (w,(j-1)h_1)$ for $i=1,\ldots,m$ and $j=1,\ldots,n$.

For $1\leq k \leq m+n-1$, the \emph{$k$th level} of $D$ is the set of edges $p_iq_j$ with $i+j=k+1$; see Figure~\ref{fig-levels}.
Note that the levels of $D$ partition the edge set of $K_{m,n}$ and that the $k$th level of $D$ contains $\min\{k,m,n,m+n-k\}$ edges.
If $D$ is regular, then, for every $1<k<m+n-1$, the edges of the $k$th level of $D$ share a unique intersection point that lies on the vertical line $\{\frac{d_1}{d_1+h_1}w\}\times \mathbb{R}$.

\begin{figure}[ht]
\centering
\includegraphics{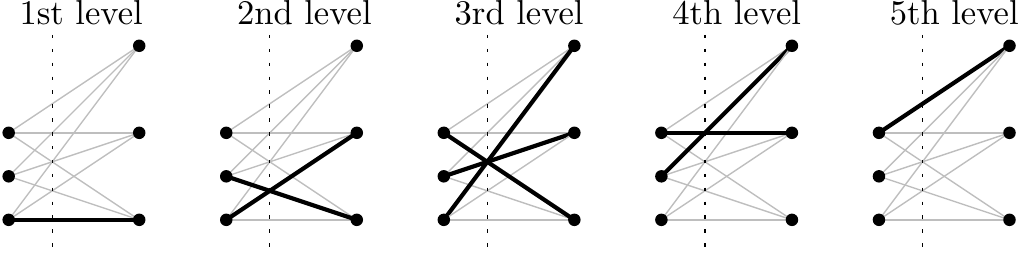}
\caption{The partitioning of the edges of a regular drawing of $K_{3,3}$ into levels. The edges in the same level are denoted by black line segments.
}
\label{fig-levels}
\end{figure}

For an integer $l \geq 2$, an ordered $l$-tuple $(p_{i_1}q_{j_1},\ldots,p_{i_l}q_{j_l})$ of edges of $D$ is \emph{uniformly crossing} if we have $0<i_2-i_1=\cdots=i_l-i_{l-1}$ and $j_2-j_1=\cdots=j_l-j_{l-1}<0$.
In particular, a set of edges forming a level of $D$, ordered by their decreasing slopes, is uniformly crossing.
Note that if $(p_{i_1}q_{j_1},\ldots,p_{i_l}q_{j_l})$ is uniformly crossing, then the edges $\pi(p_{i_1})\pi(q_{j_1}),\ldots,\pi(p_{i_l})\pi(q_{j_l})$ of the regula\-rization of $D$ share a common intersection point, which we call the \emph{meeting point of $(p_{i_1}q_{j_1},\ldots,p_{i_l}q_{j_l})$}.
Note that this point does not necessarily lie in $D$; see Figure~\ref{fig-meetingPoint}.
In the other direction, if $D$ is regular and $(e_1,\ldots,e_l)$ is a maximal set of edges of $D$ that share a common intersection point and are ordered by their decreasing slopes, then  $(e_1,\ldots,e_l)$ is uniformly crossing.

\begin{figure}[ht]
\centering
\includegraphics{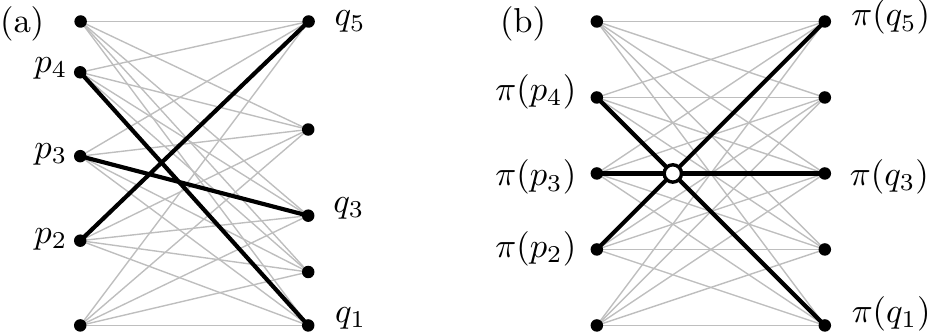}
\caption{(a) A uniformly crossing triple $T=(p_2q_5,p_3q_3,p_4q_1)$ in a bipartite drawing $D$ of $K_{5,5}$. (b) The meeting point of $T$ in the regularization of $D$.
}
\label{fig-meetingPoint}
\end{figure}

Let $\varepsilon>0$ be a real number. 
We say that $D$ is \emph{$\varepsilon$-dilated} if we have $d_1<\cdots<d_{m-1}<(1+\varepsilon)d_1$ and $h_1<\cdots<h_{n-1}<(1+\varepsilon)h_1$; see part~(a) of Figure~\ref{fig-cap}.

In a geometric drawing $D'$ of a (not necessarily bipartite) graph, let $(e_1,\ldots,e_l)$ be an ordered $l$-tuple of edges of $D'$ with finite slopes 
and such that $e_i$ and $e_{i+1}$ share an intersection point $r_i$ for $i=1,\ldots,l-1$.
We say that $(e_1,\ldots,e_l)$ \emph{forms a cap}, if $x(r_1)<\cdots < x(r_{l-1})$ and the slopes of $e_1,\ldots,e_l$ are strictly decreasing.
For every $i$, the segment $e_i$ is the graph of a linear function $f_i \colon \pi(e_i) \rightarrow \mathbb{R}$, where $\pi(e_i)$ is the projection of $e_i$ on the $x$-axis.
The \emph{lower envelope} of $(e_1,\ldots,e_l)$ is the graph of the function defined on the union of these intervals as the pointwise minimum 
of the functions $f_i$ and is undefined elsewhere.
That is, the lower envelope is a union of finitely many piecewise linear curves, called \emph{components}.
The \emph{cap} $C$ formed by $(e_1,\ldots,e_l)$ is the component of the lower envelope that contains $r_1,\ldots,r_{l-1}$.
The points $r_i$ are \emph{vertices of $C$} and $e_1\cap C,\ldots,e_l\cap C$ are \emph{edges of~$C$}; see part~(b) of Figure~\ref{fig-cap}.
A cap $C$ is \emph{good in $D'$}, if the edges of $C$ are incident to the same bounded face of $D'$ or if $C$ has only one edge.
If $D'$ is bipartite and the edges of one of its levels form a cap $C$, then we call $C$ a \emph{level-cap of $D'$}.

\begin{figure}[ht]
\centering
\includegraphics{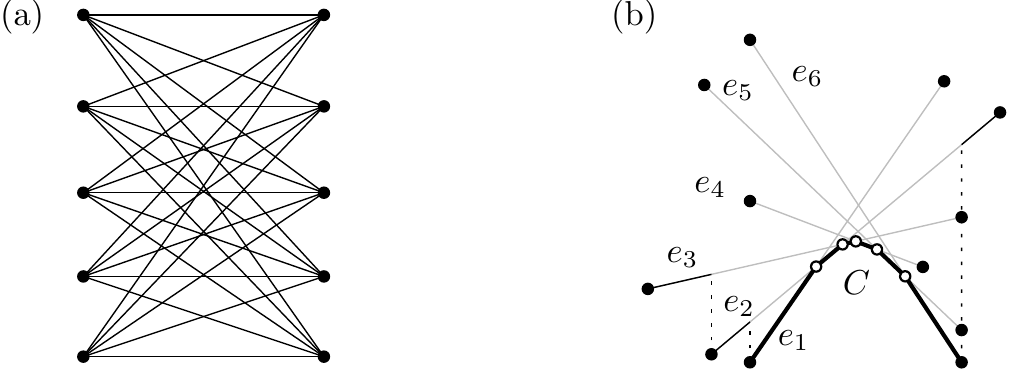}
\caption{(a) An $\varepsilon$-dilated drawing of $K_{5,5}$. (b) An example of a cap $C$ formed by $(e_1,\dots,e_6)$ with vertices denoted by empty circles.
Note that the lower envelope of $(e_1,\dots,e_6)$ is not connected.}
\label{fig-cap}
\end{figure}

The following lemma is crucial in the proofs of all our main results.

\begin{lemma}
Let $D$ be a bipartite drawing of $K_{m,n}$.
\label{lem-cap}
\begin{enumerate}[label=(\roman*)]
\item\label{item-lemCap1} If $D$ satisfies $d_1< \cdots < d_{m-1}$ and $h_1 < \cdots< h_{n-1}$, then, for every $l\geq 2$, every uniformly crossing $l$-tuple of edges of $D$ forms a cap.
\item\label{item-lemCapNew}
For all $w,d_1,h_1,\delta \in \mathbb{R}^+$ and $m,n \in \mathbb{N}$, there is an $\varepsilon_0 = \varepsilon_0(m,n,w,d_1,h_1,\delta) > 0$ such that for every positive $\varepsilon \le \varepsilon_0$ the following statement holds.
If $D$ is an $\varepsilon$-dilated bipartite drawing of $K_{m,n}$ with width $w$, left step $d_1$, and right step $h_1$, then the intersection point between any two edges $p_iq_j$ and $p_{i'}q_{j'}$ of $D$ lies in distance less than $\delta$ from the intersection point $\pi(p_i)\pi(q_j) \cap \pi(p_{i'})\pi(q_{j'})$.
\item\label{item-lemCap2}{For all $w,d_1,h_1 \in \mathbb{R}^+$ and $m,n \in \mathbb{N}$, there is an $\varepsilon_0 = \varepsilon_0(m,n,w,d_1,h_1) > 0$ such that for every positive $\varepsilon \le \varepsilon_0$  the following statement holds.
If $D$ is an $\varepsilon$-dilated bipartite drawing of $K_{m,n}$ with width $w$, left step $d_1$, and right step $h_1$, then for every $l\geq 2$ every uniformly crossing $l$-tuple of edges of $D$ forms a good cap in $D$.}
\end{enumerate}
\end{lemma}

\begin{proof}
For part~\ref{item-lemCap1}, let $(e_1,\ldots,e_l)$ be a uniformly crossing $l$-tuple of edges of $D$ with $e_k\colonequals p_{i_k}q_{j_k}$ for every $k=1,\ldots,l$.
Consider edges $e_k$, $e_{k+1}$, $e_{k+2}$ and let $r_k$ and $r_{k+1}$ be the points $e_k \cap e_{k+1}$ and $e_{k+1}  \cap e_{k+2}$, respectively.
The points $r_k$ and $r_{k+1}$ exist, as $y(p_{i_k})<y(p_{i_{k+1}})<y(p_{i_{k+2}})$ and $y(q_{j_{k+2}})<y(q_{j_{k+1}})<y(q_{j_k})$.

Consider the midpoint $p$ of $p_{i_k}p_{i_{k+2}}$ and the midpoint $q$ of $q_{j_k}q_{j_{k+2}}$.
Since $(e_1,\ldots,e_l)$ is uniformly crossing and $d_1< \cdots < d_{m-1}$ and $h_1< \cdots< h_{n-1}$, we have $y(p_{i_{k+1}}) < y(p)$ and $y(q_{j_{k+1}}) < y(q)$; see part~(a) of Figure~\ref{fig-lemma}.
The edges $pq$, $e_k$, and $e_{k+2}$ share a common point that lies above $e_{k+1}$.
Since $r_k$ and $r_{k+1}$ lie on $e_{k+1}$, we obtain $x(r_k)<x(r_{k+1})$.
The slopes of $e_k$, $e_{k+1}$, $e_{k+2}$ are strictly decreasing, thus $(e_1,\ldots,e_l)$ forms a cap.

\begin{figure}[ht]
\centering
\includegraphics{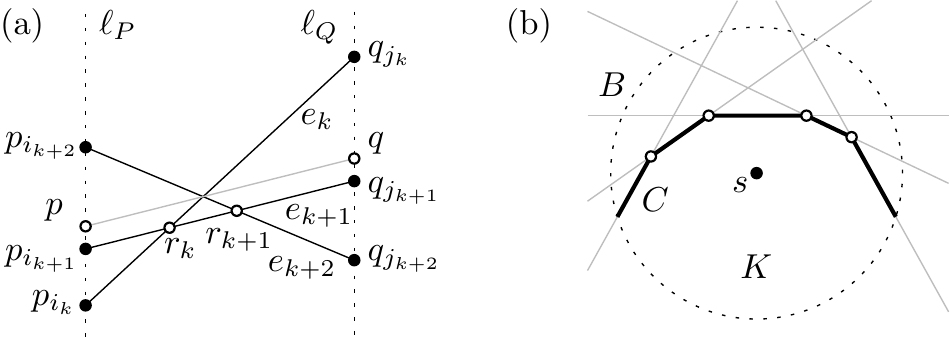}
\caption{ (a) A situation in the proof of part~\ref{item-lemCap1} of Lemma~\ref{lem-cap}. (b) A situation in the proof of part~\ref{item-lemCap2} of Lemma~\ref{lem-cap}.}
\label{fig-lemma}
\end{figure}

Part~\ref{item-lemCapNew} follows from the fact that for fixed $w$, $d_1$, $h_1$, all $\varepsilon'$-dilated drawings of $K_{m,n}$ with width $w$, left step $d_1$, and right step $h_1$ converge to their common regularization as $\varepsilon'>0$ tends to zero.

We now show part~\ref{item-lemCap2}.
We let $\delta_{m,n}(w,d_1,h_1)=\delta>0$ be the half of the minimum distance between two intersection points of the regular drawing $D'$ of $K_{m,n}$ with width $w$, left step $d_1$, and right step $h_1$.
We take $\varepsilon = \varepsilon_0(m,n,w,d_1,h_1,\delta)$ from part~\ref{item-lemCapNew} and we let $D$ be an $\varepsilon$-dilated drawing of $K_{m,n}$ with width $w$, left step $d_1$, and right step $h_1$.
According to~\ref{item-lemCap1}, every uniformly crossing $l$-tuple $(e_1,\ldots,e_l)$ of edges of $D$ forms a cap.

Let $s$ be the meeting point of $\{\pi(e_1),\ldots,\pi(e_l)\}$ in the regular drawing $D'$.
Let $B$ be the open disc with the center $s$ and the radius $\delta$.
Let $R$ be the set of edges that intersect some other edge inside $B$.
By the choice of $\delta$, $R$ is the set of edges that correspond to the edges of the regular drawing $D'$ that contain $s$.
In particular, $\{e_1,\ldots,e_l\} \subseteq R$ (note that $\{e_1,\dots,e_l\}$ might be, for example, a subset of a level of $D$ and thus we do not necessarily have $\{e_1,\ldots,e_l\} = R$).
By part~\ref{item-lemCap1}, the edges of $R$ form a cap.
All the vertices of this cap lie inside $B$.
Since no edge outside $R$ crosses any edge of $R$ inside $B$, there is a face $K$ in $D$ incident to all edges of $R$.
Since $l \ge 2$, $R$ contains at least two edges and thus $K$ is bounded; see part~(b) of Figure~\ref{fig-lemma}.
\end{proof}

\section{Proof of Theorem~\ref{thm-BipartiteUpperBound}}
\label{sec:proofThmBipartite}

Let $G\subseteq K_{m,n}$ be a bipartite graph and $\overline{G}$ be its complement.
Using Lemma~\ref{lem-cap}, we can easily show $\obsc(\overline{G}) \leq m+n-1$.
Let $\varepsilon>0$ be chosen as in part~\ref{item-lemCap2} of Lemma~\ref{lem-cap} for $K_{m,n}$ and $w=d_1=h_1=1$.
Consider an $\varepsilon$-dilated drawing $D$ of $K_{m,n}$ with $w=d_1=h_1=1$, $p_1=(0,0)$, and $q_1=(1,0)$.
Since edges of every level of $D$ are uniformly crossing, part~\ref{item-lemCap2} of Lemma~\ref{lem-cap} implies that the edges of the $k$th level of $D$ form a good level-cap $C_k$ in $D$ for every $1\leq k \leq m+n-1$.
That is, there is a bounded face $F_k$ of $D$ such that each edge of $C_k$ is incident to $F_k$ or $C_k$ contains only one edge.

For every integer $k$ such that $C_k$ contains a non-edge of $\overline{G}$, we construct a single convex obstacle~$O_k$.
If $C_k$ contains only one edge $e$, the obstacle $O_k$ is an arbitrary inner point of $e$.
Otherwise every edge $p_iq_{k+1-i}$ of the $k$th level of $D$ shares a line segment $s_k^i$ of positive length with $F_k$.
The obstacle $O_k$ is defined as the convex hull of the midpoints of the line segments $s_k^i$ where $p_iq_{k+1-i}$ is not an edge of $\overline{G}$; see Figure~\ref{fig-obstacle}.

\begin{figure}[ht]
\centering
\includegraphics{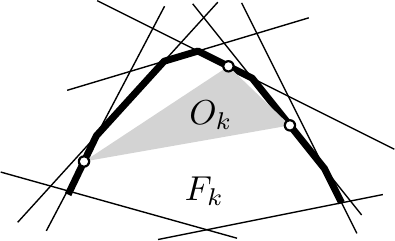}
\caption{Placing a convex obstacle $O_k$ that blocks three edges of $K_{m,n}$. }
\label{fig-obstacle}
\end{figure}

The levels partition the edge set of $K_{m,n}$, therefore we block every non-edge of $\overline{G}$.
Since every bounded face of $D$ is convex, we have $O_k \subseteq F_k$. 
Therefore no edge of $\overline{G}$ is blocked and we obtain an obstacle representation of $\overline{G}$.
In total, we produce at most $m+n-1$ obstacles.

To show $\obs_c(G) \leq m+n-1$, we proceed analogously as above, except the vertices of $D$ are suitably perturbed before obstacles $O_k$ are defined, which allows to add two (long and skinny) convex obstacles $O_P$ and $O_Q$ blocking all the edges $p_ip_{i'}$ and $q_jq_{j'}$, respectively. 
The addition of the obstacles $O_P$ and $O_Q$ may be compensated by using a single convex obstacle to block non-edges in the first and the second level and in the $(m+n-2)$nd and the $(m+n-1)$st level.

If $G$ is a split graph, then we add only one of the obstacles $O_P$ and $O_Q$ depending on whether $P$ or $Q$, respectively, represents an independent set in $G$.

\section{Proof of Theorem~\ref{thm-GeneralUpperBound}}
\label{sec:proofThmGeneral}

\begin{proof}
We show that the convex obstacle number of every graph $G$ on $n$ vertices is at most $n\lceil \log{n}\rceil-n+1$.
The high-level overview of the proof is as follows.
We partition the edges of~$G$ to edge sets of $O(n)$ induced bipartite subgraphs of $G$ by iteratively partitioning the vertex set of $G$ into two (almost) equal parts and considering the corresponding induced bipartite subgraphs of $G$.
For every $j=0,\ldots,\lfloor \log{n}\rfloor$, the number of such bipartite subgraphs of size about $n/2^j$ is $2^j$.
Then we construct an obstacle representation of $G$ whose restriction to every such bipartite subgraph resembles the obstacle representation from the proof of Theorem~\ref{thm-BipartiteUpperBound}.
Since the obstacle representation of every bipartite subgraph of size about $n/2^j$ uses about $n/2^j$ obstacles, we have $O(n\log{n})$ obstacles in total.

Let $m$ be a positive integer and let $S$ be a finite set of $m$ points on a vertical line.
The \emph{bottom half of $S$} is the set of the first $\lceil m/2 \rceil$ points of $S$ in the ordering of $S$ by increasing $y$-coordinates.
The set of $\lfloor m/2\rfloor$ remaining points of $S$ is called the \emph{top half of $S$}.

Let $\delta \in (0,1/8)$ and let $\varepsilon>0$ be a number smaller than $\varepsilon_0$ from part~\ref{item-lemCap2} of~Lemma~\ref{lem-cap} for $K_{n,n}$ and $w=d_1=h_1=1$.
We also assume that $\varepsilon$ is smaller than $\varepsilon_0$ from part~\ref{item-lemCapNew} of~Lemma~\ref{lem-cap} for $K_{n,n}$, $w=d_1=h_1=1$, and $\delta$.
Let $D$ be an $\varepsilon$-dilated bipartite drawing of $K_{n,n}$ with width, left step, and right step equal to 1 and with $d_i=h_i$ for every $i=1,\ldots,n-1$.
We let $P\colonequals\{p_1,\ldots,p_n\}$ and $Q\colonequals\{q_1,\ldots,q_n\}$ be the color classes of $D$ ordered by increasing $y$-coordinates such that $p_1=(0,0)$ and $q_1=(1,0)$.
By part~\ref{item-lemCap2} of Lemma~\ref{lem-cap}, edges of each level of $D$ form a good cap in $D$.
For the rest of the proof, the $y$-coordinates of all points remain fixed.

Let $L$ be a level of $D$ with at least two edges.
Since the edges of $L$ form a good cap, there is a face $F_L$ of $D$ such that all edges of $L$ are incident to $F_L$. 
If $L$ contains an edge $p_iq_i$ for some $1\leq i \leq n$, we let $\ell_L$ be the horizontal line containing $p_iq_i$.
Otherwise $L$ contains edges $p_iq_{i-1}$ and $p_{i-1}q_i$ for some $1 < i \leq n$ and we let $\ell_L$ be the horizontal line $\mathbb{R} \times \{(y(p_{i-1})+y(p_i))/2\}$.
No vertex of the level-cap formed by the edges of $L$ lies strictly above $\ell_L$ and thus whenever two edges of $L$ intersect above $\ell_L$, one of them has a positive slope and the other one negative.
Thus there is $\alpha_L=\alpha_L(\varepsilon)>0$ such that every edge of $L$ with positive slope is incident to a part of $F_L$ that is in the vertical distance larger than $\alpha_L$ below $\ell_L$; see Figure~\ref{fig-alfa}.
We choose $\alpha=\alpha(\varepsilon)$ to be the minimum of $\alpha_L$ over all levels $L$ of $D$ with at least two edges.
We may assume $\alpha \leq \delta$.
Note that $\alpha$ depends only on~$\varepsilon$.

\begin{figure}[ht]
\centering
\includegraphics{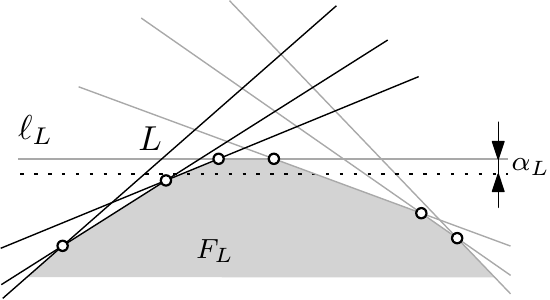}
\caption{All edges of a level $L$ with positive slope are incident to a part of a face $F_L$ strictly below $\ell_L$.
Empty circles represent vertices of the level-cap formed by the edges of the level $L$.
Grey area represents the face $F_L$.}
\label{fig-alfa}
\end{figure}

As the first step in our construction, we let $D_1$ be the drawing obtained from $D$ by removing the top half of $P$ and the bottom half of $Q$.
We forget about the part $D \setminus D_1$, as it will not be used in the obstacle representation of $G$.
We use $P^1_1$ and $P^2_1$ to denote the left and the right color class of $D_1$, respectively.
We map the vertices of $G$ to the vertices of $D_1$ arbitrarily.
Let $\mathcal{C}_1$ be the set of the level-caps of $D_1$.
Since every level-cap in $D$ is good in $D$, every cap in $\mathcal{C}_1$ is good in $D_1$.
Let $V_1^1$ be the vertical strip between  $P^1_1$ and $P^2_1$.

We now give a brief overview of the next steps.
The drawing $D_1$ is the first step towards making an obstacle representation of $G$.
In fact, we can now block a large portion of non-edges of $G$ by placing obstacles in $D_1$ as in the proof of Theorem~\ref{thm-BipartiteUpperBound}.
Then we take care of the edges between vertices in the left color class $P^1_1$ of $K_{\lceil n/2 \rceil,\lfloor n/2 \rfloor}$ as follows (edges between vertices in the right color class $P^2_1$ of $K_{\lceil n/2 \rceil,\lfloor n/2 \rfloor}$ are dealt with analogously).
We slightly shift the top half of $P^1_1$ horizontally to the right.
Only some of the edges of a copy of $K_{\lceil \lceil n/2 \rceil/2 \rceil,\lfloor \lceil n/2 \rceil/2 \rfloor}$ between the top and the bottom half of $P^1_1$ belong to $G$.
In the same way as in the bipartite case, we place convex obstacles along the level-caps of this copy.
To take care of the edges between vertices in the same color class of $K_{\lceil \lceil n/2 \rceil/2 \rceil,\lfloor \lceil n/2 \rceil/2 \rfloor}$, and for each of the color classes we proceed similarly as above. 

For an integer $j$ with $2 \leq j \leq \lceil \log{n} \rceil$, we now formally describe the $j$th step of our construction.
Having chosen the drawing $D_{j-1}$, the set $\mathcal{C}_{j-1}$, and point sets $P_{j-1}^1,\ldots,P_{j-1}^{2^{j-1}}$, we define $P_j^1,\ldots,P_j^{2^j}$ as follows.
For $1\leq k \leq 2^{j-1}$, let $P_j^{2k-1}$ be the bottom half of $P_{j-1}^k$ and let $P_j^{2k}$ be the top half of $P_{j-1}^k$.
Let $\varepsilon_j>0$ be a small real number to be specified later.
If $k$ is odd and $|P_{j-1}^{k}|>1$, we move the points from $P_j^{2k}$ to the right by $\varepsilon_j$.
If $k$ is even and $|P_{j-1}^{k}|>1$, we move the points from $P_j^{2k-1}$ to the left by $\varepsilon_j$; see Figure~\ref{fig-construction} for an illustration.
This modified drawing is $D'_{j-1}$ and $\mathcal{C}_{j-1}$ is transformed into $\mathcal{C}'_{j-1}$.

\begin{figure}[ht]
\centering
\includegraphics{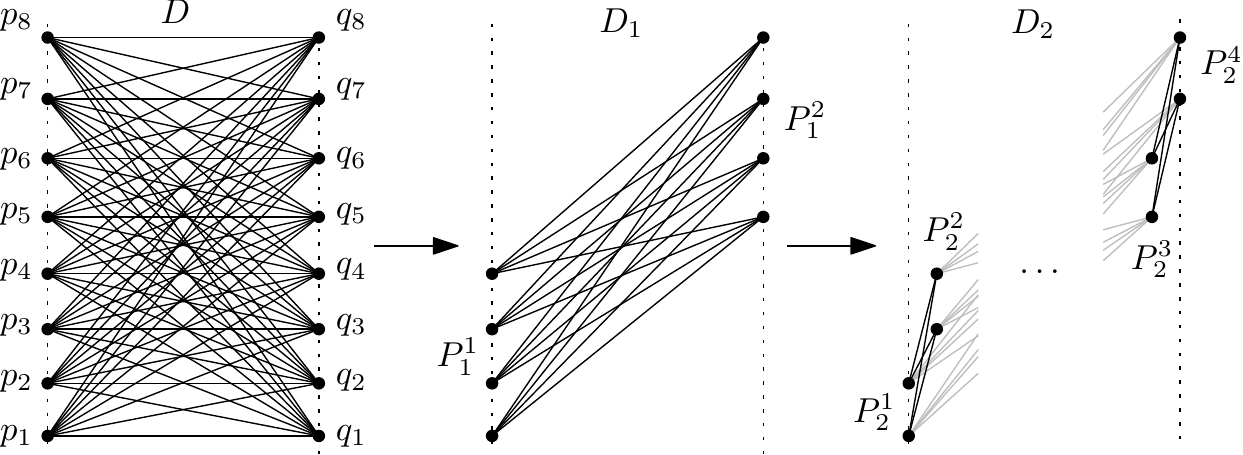}
\caption{The initial drawing $D$ and the first two iterations in the construction of the drawing of $K_8$ in the proof of Theorem~\ref{thm-GeneralUpperBound}.
The edges added in the current iteration are denoted by black segments. 
Note that the vertical strip between $P^1_2$ and $P^4_2$ in $D_2$ is the same as the vertical strip between $P^1_1$ and $P^2_1$ in $D_1$.}
\label{fig-construction}
\end{figure}

For $1\leq k \leq 2^{j-1}$, we add all edges between points from $P_j^{2k-1}$ and $P_j^{2k}$ to create a bipartite drawing $D_j^k$ of $K_{a,b}$ for some integers $a,b$ satisfying $|a-b| \leq 1$ and $a,b \leq \lceil n/2^j \rceil$.
Let $V_j^k$ be the vertical strip between $P_j^{2k-1}$ and $P_j^{2k}$.
That is, $V_j^k$ contains $D_j^k$.
We let $\mathcal{C}_j$ be the union of $\mathcal{C}'_{j-1}$ with a set of level-caps of the drawings $D_j^k$ for $1\leq k \leq 2^{j-1}$.
We also set $D_j\colonequals D_j^1 \cup \cdots \cup D^{2^{j-1}}_j \cup D'_{j-1}$.

We choose $\varepsilon_j$ small enough so that each good cap $C\in\mathcal{C}_{j-1}$ is transformed to a good cap $C' \in \mathcal{C}'_{j-1}$.
Such $\varepsilon_j$ exists, as every geometric drawing of a graph is compact and the distance of two points is a continuous function.

We also choose $\varepsilon_j$ small enough so that the following holds for every edge $e$ of $D'_{j-1}$. Let $\gamma \le j-1$ be the smallest index such that $e \in D'_{\gamma}$. For every vertical strip $V_{j'}^{k'}$ intersecting $e$ and with $\gamma < j'\le j$ and $1 \leq k'\leq 2^{j'-1}$, the portion of $e$ in the vertical strip $V_{j'}^{k'}$ is contained in the horizontal strip $\mathbb{R}\times (y(p)-\alpha,y(p)+\alpha)$ for some endpoint $p$ of $e$.
This can be done, as $\alpha$ depends only on $\varepsilon$ and the endpoints of $e$ move by at most $\varepsilon_j$.
We also use the facts that the vertical strips $V_{j'}^{k'}$, for $1 \le j'< j$ and $1 \leq k'\leq 2^{j'-1}$, do not change during the translations by $\varepsilon_j$ and that each vertical strip $V_{j'}^{k'}$, for $2 \leq  j' \leq j$ and $1 \leq k'\leq 2^{j'-1}$, is contained in the vertical strip $V_{j'-1}^{\lceil k'/2 \rceil}$; see Figure~\ref{fig-shift} for an illustration.

\begin{figure}[ht]
\centering
\includegraphics{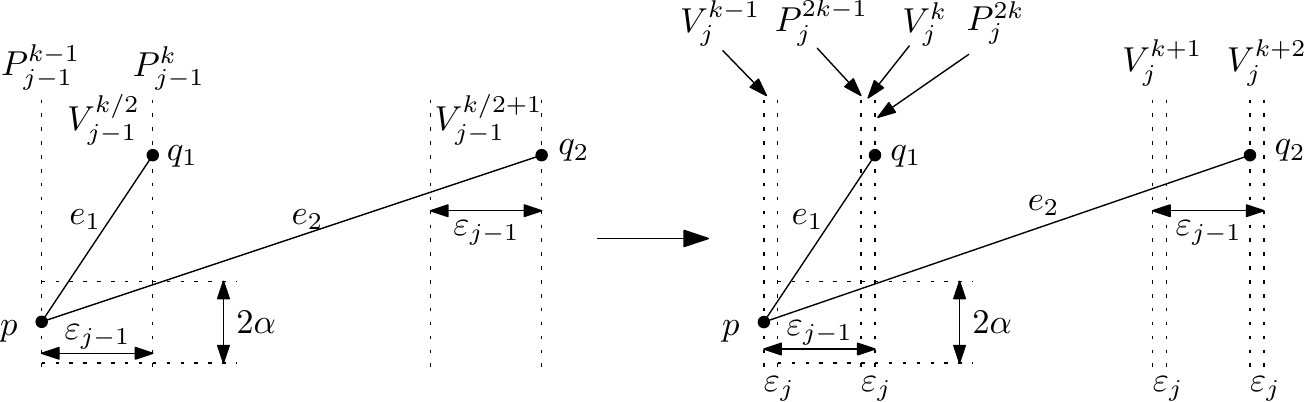}
\caption{An example of conditions posed on the edges $e \in D'_{j-1}$ while choosing $\varepsilon_j$. In the example, $j\ge 2$, $k$ is even, $e_1 \in D'_{j-1}$ and $e_2 \in D'_{j-2}$.
The portions of $e_1$ in $V_j^{k-1}$ and $V_j^{k}$ are contained in the horizontal strips of height $2 \alpha$ around $p$ and $q_1$, respectively.
The portions of $e_2$ in $V_{j-1}^{k/2}$, $V_j^{k-1}$, and $V_j^{k}$ are contained in the horizontal strips of height $2 \alpha$ around $p$.
The portions of $e_2$ in $V_{j-1}^{k/2+1}$, $V_j^{k+1}$, and $V_j^{k+2}$ are contained in the horizontal strips of height $2 \alpha$ around $q_2$.
}
\label{fig-shift}
\end{figure}

After $\lceil\log{n}\rceil$ steps, the drawings $D_{\lceil\log{n}\rceil}^k$ contain at most two vertices and the construction stops.
We show that we can add at most $ n\lceil \log{n}\rceil -n+1$ convex obstacles to the drawing $D_{\lceil \log{n}\rceil}$ to obtain an obstacle representation of $G$.

For $2 \leq j \leq \lceil \log{n} \rceil$ and $1 \leq k \leq 2^{j-1}$, let $f_{j,k}\colon \mathbb{R}^2\to \mathbb{R}^2$ be the affine mapping $f_{j,k}(x,y)\colonequals(x/\varepsilon_j-c_{j,k},y)$ where $c_{j,k}\in\mathbb{R}$ is chosen such that the left color class of $f_{j,k}(D^k_j)$ lies on $\{0\} \times \mathbb{R}$.
Note that the drawing $f_{j,k}(D^k_j)$ is contained in the drawing $D$ and thus edges of the levels of $f_{j,k}(D^k_j)$ form good caps in $f_{j,k}(D^k_j)$; see Figure~\ref{fig-affine}.
Since $f_{j,k}$ does not change the edge-face incidences in $D^k_j$, edges of the levels of $D^k_j$ form good caps in $D^k_j$.

Let $C$ be a level-cap of $D^k_j$.
Let $L$ be the level of $f_{j,k}(D^k_j)$ whose edges form the level-cap $f_{j,k}(C)$ of $f_{j,k}(D^k_j)$.
Edges of $L$ are also edges of a level $L'$ of $D$ and we let $\ell_C$ be the horizontal line $\ell_{L'}$.
All edges of $L$ have positive slope in $D$.
Thus it follows from the definition of $\alpha$ that there is a bounded face of $f_{j,k}(D^k_j)$ such that all edges of $f_{j,k}(C)$ are incident to the part of this face below $\ell_C$ in the vertical distance larger than $\alpha$.
Since $f_{j,k}$ does not change the $y$-coordinates, we get that for every level-cap $C$ of $D^k_j$, there is a bounded face $F_C$ of $D^k_j$ such that all edges of $C$ are incident to the part of $F_C$ that lies below $\ell_C$ in the vertical distance larger than $\alpha$.

\begin{figure}[ht]
\centering
\includegraphics{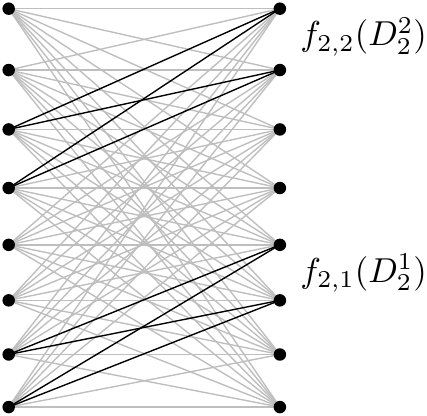}
\caption{An example of the drawings $f_{2,1}(D^1_2)$ and $f_{2,2}(D^2_2)$ (denoted by black segments) in the $\varepsilon$-dilated drawing $D$ of $K_{8,8}$.}
\label{fig-affine}
\end{figure}

By induction on $j$, $1 \leq j \leq \lceil\log{n}\rceil$, we show that every cap from $\mathcal{C}_j$ is good in $D_j$ in the $j$th step of the construction.
We already observed that this is true for $j=1$.
Suppose for a contradiction that there is a cap $C \in \mathcal{C}_j$ that is not good in $D_j$ for $j>1$.
Using the inductive hypothesis and the choice of $\varepsilon_j$, we see that $C$ is not in $\mathcal{C}'_{j-1}$.
Therefore there is a drawing $D_j^k$ for $1\leq k \leq 2^{j-1}$ such that $C$ is a level-cap of $D_j^k$.
Since $C$ is good in $D_j^k$, all edges forming $C$ are incident to a single bounded face $F_C$ of $D^k_j$. 
The drawings $D^1_j,\ldots,D_j^{2^{j-1}}$ are contained in pairwise disjoint vertical strips, thus $C$ is good in $D_j^1 \cup \cdots \cup D_j^{2^{j-1}}$.

It follows from our choice of $\delta$ and $\varepsilon$ and from part~\ref{item-lemCapNew} of Lemma~\ref{lem-cap} that all edges of~$C$ are incident to $F_C$ in a $1/8$-neighborhood of $\ell_C$.
Moreover, all edges of~$C$ are incident to the part of $F_C$ that  lies below $\ell_C$ in the vertical distance larger than $\alpha$.
Therefore all edges of~$C$ are incident to the part of~$F_C$ that is in the horizontal strip $H$ between the horizontal lines $\ell_C - 1/8$ and $\ell_C-\alpha$.
Because $C$ is not good in~$D_j$, some edge $e$ of $D'_{j-1}$ intersects $H$ in the vertical strip $V_j^k$.
By the choice of $\varepsilon_j$, the portion of $e$ inside $V_j^k$ is contained in the horizontal strip $\mathbb{R}\times (y(p)-\alpha,y(p)+\alpha)$ for an endpoint $p$ of $e$.
However, the horizontal line $\ell_C$ is by definition either at distance at least $1/2$ from every point of $P$ or it contains some edge $p_iq_i$ with $1 \leq i \leq n$.
In any case, since $\alpha \le \delta < 1/8$, the strip $H$ is disjoint with all the horizontal strips $\mathbb{R}\times (y(p)-\alpha,y(p)+\alpha)$ for all endpoints $p$ in $D'_{j-1}$ and we obtain a contradiction.

Let $\widehat{D}_j^k$ be the part of the drawing $D_{\lceil \log{n}\rceil}$ transformed from $D_j^k$ in the steps $j+1, \ldots, \lceil\log{n}\rceil$.
For every $\widehat{D}_j^k$, we place the obstacles to its good level-caps as in the first part of the proof of Theorem~\ref{thm-BipartiteUpperBound}.
Using the fact that bounded faces of every geometric drawing of $K_n$ are convex, it follows from the construction of $D_{\lceil\log{n}\rceil}$ that we obtain an obstacle representation of $G$.

The total number of obstacles that we use can be estimated recursively.
Every drawing $\widehat{D}^k_j$ is a bipartite drawing of $K_{a,b}$ for some integers $a,b$ satisfying $|a-b| \leq 1$ and $a,b \leq \lceil n/2^j \rceil$.
For such $\widehat{D}^k_j$ we use $a+b-1$ obstacles.
It follows from the construction that the maximum number of obstacles that we use for an $n$-vertex graph can be estimated from above by a function $h(n)$ that is given by a recursive formula $h(n)=h(\lceil n/2 \rceil)+h(\lfloor n/2 \rfloor)+n-1$ with $h(1)=0$.

Now it suffices to show $h(n) \leq n\lceil \log{n}\rceil -n+1$.
We proceed by induction on $n$.
The inequality is trivially satisfied for $n=1$.
For $n>1$, we have
\begin{align*}
h(n)&=h\left(\left\lceil \frac{n}{2} \right\rceil \right)+h\left(\left\lfloor \frac{n}{2} \right\rfloor \right)+n-1 \\ 
& \leq \left\lceil \frac{n}{2} \right\rceil \left\lceil\log{\left(\left\lceil \frac{n}{2} \right\rceil\right)}\right\rceil - \left\lceil\frac{n}{2} \right\rceil + 1 + \left\lfloor \frac{n}{2} \right\rfloor \left\lceil\log{\left(\left\lfloor \frac{n}{2} \right\rfloor\right)}\right\rceil - \left\lfloor\frac{n}{2} \right\rfloor + 1 + n - 1.
\end{align*}
We have $\lceil\log \lceil\alpha\rceil \rceil = \lceil\log \alpha \rceil$ for every real number $\alpha \ge 1$ and thus $\lceil\log(\lceil n/2 \rceil)\rceil = \lceil\log(n/2)\rceil = \lceil\log n\rceil - 1$, since $n \ge 2$.
Therefore
\begin{align*}
h(n)& \leq \left\lceil \frac{n}{2} \right\rceil \left(\left\lceil\log{n}\right\rceil-1\right) + \left\lfloor \frac{n}{2} \right\rfloor  \left(\left\lceil\log{n}\right\rceil-1\right) + 1 = n\lceil \log{n}\rceil -n+1.
\end{align*}
This finishes the proof of Theorem~\ref{thm-GeneralUpperBound}.\qed
\end{proof}

An alternative construction, in which the underlying point set is a variant of the well-known \emph{Horton sets}~\cite{Valtr}, can be found in the conference version of this paper~\cite{Balko+15}.

Let $P$ be a finite set of points of the plane, let $G$ be a graph with vertex set $V$, $|V|=|P|$, and let $f \colon V \to P$ be a one-to-one correspondence.
An obstacle representation of $G$ is \emph{induced by $f$} if every vertex $v$ of $G$ is represented by the point $f(v) \in P$.
The proof of Theorem~\ref{thm-GeneralUpperBound} gives the following stronger claim.

\begin{corollary}
\label{cor-GeneralUpperBound}
For every positive integer $n$, there is a set $P$ of $n$ points in the plane such that for every graph $G=(V,E)$ on $n$ vertices and for every one-to-one correspondence $f \colon V \to P$ there is an obstacle representation of $G$ induced by $f$ with at most $n\lceil \log{n} \rceil-n+1$ convex obstacles.\qed
\end{corollary}

We also note that the vertices of the drawing $D_{\lceil\log{n}\rceil}$ can be perturbed so that the vertices of the resulting drawing are in general position.
When the perturbation is small enough, the obstacles can be modified to block exactly the same set of edges that they blocked before the perturbation.

\section{Proof of Theorem~\ref{thm-chromaticUpperBound}}
\label{sec:proofThmChromatic}

To prove Thoerem~\ref{thm-chromaticUpperBound}, we start with the same $\varepsilon$-dilated drawing $D$ of $K_{n,n}$ and use $P$ and $Q$ to denote the left and the right color class of $D$, respectively.
Again, by part~\ref{item-lemCap2} of Lemma~\ref{lem-cap}, edges of each level of $D$ form a good cap in $D$.

For a finite set $S$ of points on a vertical line $\ell$, an \emph{interval in $S$} is a subset $I$ of $S$ such that there is no point of $S\setminus I$ lying between two points of $I$ on $\ell$.
If $S$ is partitioned into intervals $I_1,\ldots,I_m$, then the \emph{bottom part of $S$ with respect to $I_1,\ldots,I_m$} is the set of points of $S$ that are contained in the intervals $I_1,\ldots,I_{\lceil m/2\rceil}$.
The set of remaining points of $S$ is called the \emph{top part of $S$ with respect to $I_1,\ldots,I_m$}.

Let $c$ be a subcoloring of $G$ with $\chi_s(G)$ colors.
We map the vertices of $G$ to the points of $P$ such that the color classes of $c$ form intervals in~$P$ ordered by increasing $y$-coordinates and such that in each color class the vertex sets of the disjoint cliques also form  intervals in~$P$ ordered by increasing $y$-coordinates.
We map the vertices of $G$ to $Q$ in the same order by increasing $y$-coordinate.
Let $P^1_1$ be the bottom part of $P$ and let $P^2_1$ be the top part of $Q$ with respect to the partitioning of $P$ and $Q$ into the color classes of $c$.
We let $D_1$ be the drawing of the complete bipartite graph with the bottom part of $P$ as the left color class of $D_1$ and with the top part of $Q$ as the right color class of $D_1$.

The construction of the obstacle representation of $G$ then proceeds analogously as in the proof of Theorem~\ref{thm-GeneralUpperBound} with only one difference.
Instead of partitioning the set $P^k_{j-1}$ into the bottom and the top half, we partition $P^k_{j-1}$ into the bottom and the top part with respect to the partitioning of $P^k_{j-1}$ into the color  classes of $c$.
If all points of $P^k_{j-1}$ are contained in the same color class of $c$, then we do not partition $P^k_{j-1}$ at all.

For every drawing $D^k_j$, we place the obstacles to block non-edges of $G$ in the same way as in the proof of Theorem~\ref{thm-GeneralUpperBound}.
We place obstacles only to block non-edges  of $G$ between vertices from different color classes of $c$.

We now estimate the number of obstacles used so far.
The number of steps is at most $\lceil \log{\chi_s(G)}\rceil$, as the number of intervals in the partitioning of $P^k_j$ is at most $\lceil \chi_s(G) / 2^j \rceil$.
We place at most $n-1$ obstacles in every step, as the number of obstacles used in a drawing $D^k_j$ is less than the number of vertices in $D^k_j$  and the total number of vertices in the drawings $D^k_j$ is at most $n$ for a fixed $j$.
Thus the total number of obstacles is at most $(n-1)\lceil \log{\chi_s(G)}\rceil$.

It remains to block non-edges within each color class of the subcoloring~$c$.
Let $P_j^k$ be a set that forms a color class of~$C$.
We define a single-point obstacle $O_{u,v}$ for every pair $\{u,v\}$ of consecutive vertices of $P^k_j$ that are not contained in a common clique of the color class $P^k_j$ such that $O_{u,v}$ blocks no edge of~$G$.
Since the cliques in the color class $P^k_j$ form intervals  and all vertices from each color class lie on a common line, the obstacles $O_{u,v}$ exist.
This gives us less than $n$ additional obstacles, as there is at most $n-1$ such pairs $\{u,v\}$ of consecutive vertices of $G$.

Note that the construction from the proof of Theorem~\ref{thm-GeneralUpperBound} is a special case of this construction, in which the color classes of $c$ consist of a single point.
The vertices of the final drawing in the obstacle representation of $G$ are not in general position.
However, they can be perturbed to be in general position and the obstacles can be modified to keep blocking the same set of edges, in particular, the one-point obstacles can be replaced by short horizontal segments.

\begin{remark}
There are graphs $G$ for which the bound $\obs_c(G) \leq (n-1)(\lceil\log{\chi_s(G)}\rceil + 1)$ from Theorem~\ref{thm-chromaticUpperBound} is significantly better than $\obs_c(G) \leq (n-1)(\lceil\log{\chi_0}\rceil + 1)$ where $\chi_0 \colonequals \min\{\chi(G),\chi(\overline{G})\}$.
For example, if $G$ is a disjoint union of $t$ cliques, each of size $t$ for some positive integer $t$, then the first bound gives $\obs_c(G) \leq n-1$ for $n\colonequals t^2$, as $\chi_s(G) = 1$.
On the other hand, we have $\chi(G) = t = \chi(\overline{G})$ and thus the latter bound gives only $\obs_c(G) \leq (n-1)\lceil \log{t}\rceil +1 = \Theta(n\log{n})$.
\end{remark}

\begin{remark}
Note that we can choose vertices and obstacles in the obstacle representation $\mathcal{O}$ of a graph $G=(V,E)$ from Theorem~\ref{thm-chromaticUpperBound} in such a way that every non-edge of $G$ intersects exactly one obstacle and it intersects the obstacle in exactly one point.
In particular, if we add a non-edge $e$ of a graph $G$ to $G$, then by removing the point $e \cap O$ from the obstacle $O$ that intersects $e$, we construct an obstacle representation of $G+e \colonequals (V,E \cup \{e\})$ with at most one additional obstacle when compared to $\mathcal{O}$.

Motivated by this observation, we pose the following problem.
If $G$ is a graph and $e$ is a non-edge of $G$, how much larger can $\obs(G+e)$ be when compared to $\obs(G)$?
The same question can be also asked for the convex obstacle number.
Note that $\obs(G+e) \geq \obs(G)-1$ and $\obs_c(G+e) \geq \obs_c(G)-1$ for every graph $G$ and every non-edge $e$ of $G$.
\end{remark}

\section{Proof of Theorem~\ref{thm-count-lb}}
\label{sec:proofThmCount-lb}

Let $h$ and $n$ be given positive integers with $h < n$.
We show that the number $g(h,n)$ of  labeled $n$-vertex graphs of obstacle number at most $h$ is at least $2^{\Omega(hn)}$.

For a point set $P \subseteq \mathbb{R}^2$ in general position, let $e(h,P)$ be the maximum integer for which there is a set $\mathcal{F}$ of at most $h$ bounded faces of $K_P$ and a set of $e(h,P)$ edges of $K_P$ that are incident to at least one face from $\mathcal{F}$. 
Let $e(h,n)$ be the maximum of $e(h,P)$ over all sets $P$ of $n$ points in the plane in general position.

\begin{claim}
We have $g(h,n) \geq 2^{e(h,n)}$.
\end{claim}

To prove the claim, let $P$ be a set of $n$ points in the plane in general position for which $e(h,P)=e(h,n)$.
Let $\mathcal{F}$ be the set of at most $h$ bounded faces of $K_P$ such that $e(h,n)$ edges of $K_P$ are incident to at least one face from $\mathcal{F}$.
For a face $F \in \mathcal{F}$, let $E_F$ denote the set of edges of $K_P$ that are incident to $F$.
We use $G$ to denote the graph with vertex set $P$ and with two vertices connected by an edge if and only if the corresponding edge of $K_P$ is not in $\cup_{F \in \mathcal{F}}E_F$.

We show that every subgraph $G'$ of $K_P$ containing $G$ satisfies $\obs(G') \leq h$.
The claim then follows, as the number of such subgraphs $G'$ is $2^{e(h,n)}$.

Let $G'$ be a subgraph of $K_P$ such that $G\subseteq G'$.
For every face $F \in \mathcal{F}$, we define a convex obstacle $O_F$ as the convex hull of midpoints of line segments $e \cap F$ for every $e \in E_F$ that represents a non-edge of $G'$.
Note that, since all bounded faces of $K_P$ are convex, the obstacle $O_F$ is contained in $F$ and thus $O_F$ blocks only non-edges of $G'$.
Since every non-edge of $G'$ is contained in $E_F$ for some $F \in \mathcal{F}$, we obtain an obstacle representation of $G'$ with at most $h$ convex obstacles.
This finishes the proof of the claim.\qed

Since $h<n$, the following and the previous claim give Theorem~\ref{thm-count-lb}.

\begin{claim}
For $n \geq 3$, we have $e(h,n) \geq \frac{2hn-h^2-1}{4}$.
\end{claim}

Let $\varepsilon>0$ be chosen as in part~\ref{item-lemCap2} of Lemma~\ref{lem-cap} for $K_{\lceil n/2 \rceil, \lfloor n/2 \rfloor}$ and $w=d_1=h_1=1$.
Let $D$ be an $\varepsilon$-dilated drawing of $K_{\lceil n/2 \rceil, \lfloor n/2 \rfloor}$ with $w=d_1=h_1=1$, $p_1=(0,0)$, and $q_1=(1,0)$.
By part~\ref{item-lemCap2} of Lemma~\ref{lem-cap}, the edges of the $k$th level of $D$ form a good cap $C_k$ in $D$ for every $k=1,\ldots,n-1$.

We perturb the vertices of $D$ such that the vertex set of the resulting geometric drawing $D'$ of $K_{\lceil n/2 \rceil, \lfloor n/2 \rfloor}$ is in general position.
We let $K_P$ be the geometric drawing of $K_n$ obtained from $D'$ by adding the missing edges.
Note that if the perturbation is sufficiently small, then every good cap $C_k$ in $D$ corresponds to a good cap $C'_k$ in $K_P$.

Let $\mathcal{F}\colonequals\{F_1,\ldots,F_h\}$ be the set of (not necessarily distinct) bounded faces of $K_P$ such that, for $i=1,\ldots,h$, all edges of the cap $C'_{\lfloor n/2\rfloor-\lceil h/2\rceil+i}$ are incident to $F_i$.
That is, $F_1,\ldots,F_h$ are faces incident to edges of $h$ middle caps $C'_k$; see Figure~\ref{fig-lowerBound}.
Since caps $C'_{\lfloor n/2\rfloor-\lceil h/2\rceil+i}$ are good in $K_P$ and $n \geq 3$, the faces $F_i$ exist.

\begin{figure}[ht]
\centering
\includegraphics{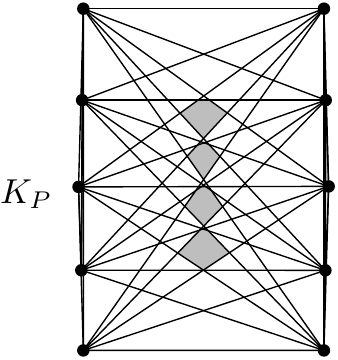}
\caption{A situation in the proof of Theorem~\ref{thm-count-lb} for $n=10$ and $h=4$. The faces $F_1,\ldots,F_h$ of $K_P$ are denoted gray.}
\label{fig-lowerBound}
\end{figure}

Every cap $C'_k$ is formed by $\min\{k, n-k\}$ edges for every $k=1,\ldots,n-1$.
Therefore, for every $i=1,\ldots,h$, the face $F_i$ is incident to at least $\min\{\lfloor n/2\rfloor-\lceil h/2\rceil+i,\lceil n/2\rceil+\lceil h/2\rceil-i\}$ edges of $K_P$.
Summing over $i=1,\ldots,h$, we obtain the following expression for the number of edges of $K_P$ incident to at least one face of $\mathcal{F}$:
\begin{align*}
e(h,n) &\ge  \sum_{i=1}^{\lceil h/2 \rceil}\left(\left\lfloor \frac{n}{2} \right\rfloor - \left\lceil \frac{h}{2} \right\rceil + i \right) + \sum_{i=\lceil h/2 \rceil + 1}^{h}\left(\left\lceil \frac{n}{2} \right\rceil+\left\lceil \frac{h}{2} \right\rceil-i\right) = \\
& = \sum_{j=1}^{\lceil h/2 \rceil}\left(\left\lfloor \frac{n}{2} \right\rfloor - j + 1 \right) + \sum_{j=1}^{\lfloor h/2 \rfloor}\left(\left\lceil \frac{n}{2} \right\rceil - j\right).
\end{align*}
If $h$ is even, we have 
\begin{align*}
e(h,n) &\ge \sum_{j=1}^{h/2}\left(\left\lfloor \frac{n}{2} \right\rfloor + \left\lceil \frac{n}{2} \right\rceil - 2j + 1 \right) 
= \frac{nh}{2} + \frac{h}{2} - 2 \sum_{j=1}^{h/2} j \\
&= \frac14 \left(2nh + 2h - h \cdot (h+2) \right) > \frac14(2nh-h^2-1).
\end{align*}
If $h$ is odd, we have 
\begin{align*}
e(h,n) &\ge \sum_{j=1}^{(h-1)/2}\left(\left\lfloor \frac{n}{2} \right\rfloor + \left\lceil \frac{n}{2} \right\rceil - 2j + 1 \right) + \left\lfloor \frac{n}{2} \right\rfloor - \frac{h-1}{2} \\
&= \frac{n(h-1)}{2} - 2 \left(\sum_{j=1}^{(h-1)/2} j\right) + \frac{h-1}{2} + \left\lfloor \frac{n}{2} \right\rfloor - \frac{h-1}{2} \\
&\ge \frac14 \left(2nh - 2n - (h-1)(h+1) + 2(n-1) \right) \\
&= \frac14(2nh-h^2-1). 
\end{align*}
This implies $e(h,n) \geq (2hn-h^2-1)/4$ and proves the claim.\qed

\section{Proof of Theorem~\ref{thm-arrangements}}
\label{sec:thm-arrangements}

For a sufficiently large constant $C$ and every sufficiently large integer $n$, we find a bipartite drawing $D$ of $K_{n,n}$ such that for every integer $M$ satisfying $Cn \leq M \leq n^4/C$ there is a set of at most $M$ faces of $D$ with complexity at least $\Omega(n^{4/3}M^{2/3})$.
When $M=Cn$, we obtain a set of at most $Cn$ faces with complexity at least $\Omega(n^2)$.
For $M \leq Cn$, it suffices to take only the faces with the highest complexity from this set to obtain the lower bound $\Omega(nM)$.
Theorem~\ref{thm-arrangements} then follows, as $D$ can be treated as an arrangement of $n^2$ line segments with $2n$ endpoints.

Let $D'$ be the regular bipartite drawing of $K_{n,n}$ with width, left step, and right step equal to 1, $p_1=(0,0)$, and $q_1=(1,0)$.
For all coprime integers $i$ and $k$ satisfying $1\leq i < k \leq n/2$, every intersection point of a uniformly crossing $l$-tuple of edges $(p_{i_1}q_{j_1},\ldots,p_{i_l}q_{j_l})$ of $D'$ with $i_2-i_1=i$ and $j_2-j_1=i-k$ is called a \emph{uniform $(i,k)$-crossing}; see part~(a) of Figure~\ref{fig-arrang}.
A point that is a uniform $(i,k)$-crossing for some integers $i$ and $k$ is called a \emph{uniform crossing}.

Note that all uniform $(i,k)$-crossings lie on the vertical line $\{\frac{i}{k}\}\times \mathbb{R}$ and that no uniform $(i,k)$-crossing is a uniform $(i',k')$-crossing for any pair $(i',k')\neq(i,k)$, as $i$ and $k$ are coprime.
Since the $y$-coordinate of every uniform $(i,k)$-crossing equals $j/k$ for some $0\leq j \leq kn-k$, the number of uniform $(i,k)$-crossings is at most $kn$.
We now show that there are at least $n^2-2in > n^2-2kn$ edges of $D'$ that contain a uniform $(i,k)$-crossing.
This follows from the condition $k \leq n/2$, as for every edge $p_{i'}q_{j'}$ of $D'$ with $i<i'\leq n-i$ and $1\leq j'\leq n$ either $p_{i'-i}q_{j'+k-i}$ or $p_{i'+i}q_{j'-k+i}$ is an edge of $D'$ and forms a uniform $(i,k)$-crossing with $p_{i'}q_{j'}$.

We choose $\varepsilon>0$ as in part~\ref{item-lemCap2} of Lemma~\ref{lem-cap} for $K_{n,n}$ and $w=d_1=h_1=1$.
Let $D$ be an $\varepsilon$-dilated drawing of $K_{n,n}$ with width, left step, and right step equal to 1, with the left lowest point $(0,0)$, and with the right lowest point $(1,0)$.
By part~\ref{item-lemCap2} of Lemma~\ref{lem-cap}, every uniformly crossing $l$-tuple of edges of $D$ forms a good cap in~$D$.
In particular, every uniform crossing~$c$ in~$D'$ is the meeting point of edges of~$D$ that form a good cap $C_c$.
Let $F_c$ be the bounded face of~$D$ such that all edges of $C_c$ are incident to $F_c$ and all vertices of $C_c$ are vertices of $F_c$; see part~(b) of Figure~\ref{fig-arrang}.

\begin{figure}[ht]
\centering
\includegraphics{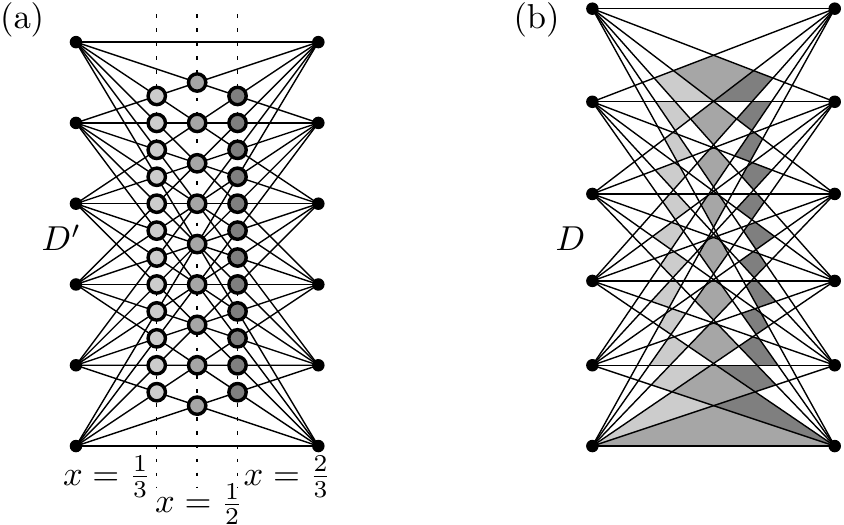}
\caption{(a) Uniform $(1,3)$-crossings (denoted light gray), $(1,2)$-crossings (gray), and $(2,3)$-crossings (dark gray) in the regular drawing $D'$ of $K_{6,6}$. (b) Faces corresponding to the $(1,3)$-crossings (light gray), to the uniform $(1,2)$-crossings (gray), and to the $(2,3)$-crossings (dark gray) in the $\varepsilon$-dilated drawing $D$.}
\label{fig-arrang}
\end{figure}

Let $c$ be the meeting point of a maximal uniformly crossing $l$-tuple $L$ of edges from $D'$ for some $l \geq 2$ .
Similarly, let $c'$ be the meeting point of a maximal uniformly crossing $l'$-tuple $L'$ of edges from $D'$ for some $l' \geq 2$.
Let $F_c$ and $F_{c'}$ be the faces of $D$ that are incident to all edges of $L$ and $L'$, respectively, in $D$.
We show that the faces $F_c$ and $F_{c'}$ of~$D$ are distinct.

Let $F'_c$ and $F'_{c'}$ be the faces that correspond to $F_c$ and $F_{c'}$ in $D'$.
It suffices to show that $F'_c$ and $F'_{c'}$ are distinct, so suppose for a contradiction that $F'_c = F = F'_{c'}$.
We can assume without loss of generality that $c$ and $c'$ lie on a common edge $e=p_rq_s$ of $D'$, since all vertices $d$ of $F$ with $F=F'_d$ are meeting points of a maximal uniformly crossing tuple of edges from $D'$ and they form an interval on $F$.
Therefore we can choose $c$ and $c'$ to be two such vertices of $F$ that share an edge of $F$.
We also assume that $c$ is to the left of $c'$ on $e$.
Then $e$ has the smallest slope among edges that contain $c$.
We let $p_{r-i}q_{s+k-i}$ be the edge that contains $c$ and that has the second smallest slope among edges that contain $c$.
Similarly, $e$ has the largest slope among edges that contain $c'$ and we let $p_{r+i'}q_{s-k'+i'}$ be the edge that contains $c'$ and that has the second largest slope among edges that contain $c'$; see part~(a) of Figure~\ref{fig-arrangFaces}.
Note that $1 \leq i < k < n$ and $1 \leq i' <  k' < n$.
Without loss of generality, we assume that $k \leq k'$.

Let $x$ be the point on $\ell_P$ above $p_r$ with $|p_rx|=|p_rp_{r-i}|$ and let $y$ be the point on $\ell_Q$ below $q_s$ with $|q_sy|=|q_sq_{s+k-i}|$.
Observe that $c$ is the common point of $e$ and $xy$.
The point $x$ is below $p_{r+i'}$, since otherwise we have $|p_rx|=i \geq i' = |p_rp_{r+i'}|$ and, since $c \in e \cap xy$, $k-i > k'-i'$.
This implies $k > k'$, which contradicts our assumption.
Therefore $x$ lies between $p_r$ and $p_{r+i'}$ on $\ell_P$ and, in particular, $x \in P$.
The point $y$ lies below $q_{s-k'+i'}$, since otherwise $y \in Q$ and $c \in e \cap xy$ implies that $xy$ gives $F'_c  \neq F'_{c'}$.
However, if $y$ is below $q_{s-k'+i'}$, then the edge $xq_{s-k'+i'}$ gives $F'_c\neq F'_{c'}$ and we obtain a contradiction; see part~(b) of Figure~\ref{fig-arrangFaces}.

\begin{figure}[ht]
\centering
\includegraphics{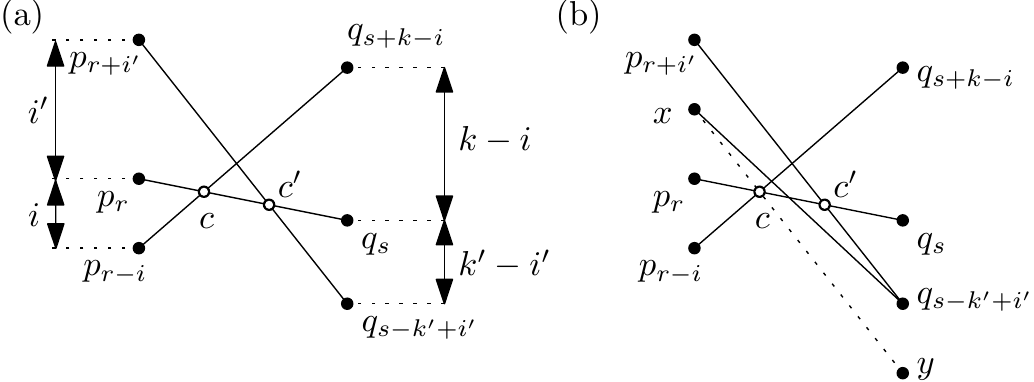}
\caption{Situations in the proof of Theorem~\ref{thm-arrangements}.}
\label{fig-arrangFaces}
\end{figure}

Let $K \leq n/2$ be a positive integer whose value we specify later.
For coprime integers $i$ and $k$ satisfying $1\leq i < k \leq K$, let $\mathcal{F}_{i,k}$ be the set of faces $F_c$ where $c$ is a uniform $(i,k)$-crossing in $D'$.
It follows from our observations that $\mathcal{F}_{i,k}$ contains at most $kn$ faces and that the complexity of $\mathcal{F}_{i,k}$ is at least $n^2-2kn$.
We let $\mathcal{F}\colonequals \bigcup_{i,k}\mathcal{F}_{i,k}$ where the union is taken over all coprime integers $i$ and $k$ satisfying $1\leq i < k \leq K$.
Then $\mathcal{F}$ contains at most 
\[
\sum_{k=2}^K\sum_{\substack{i=1\\ \gcd(i,k)=1}}^{k-1} kn 
< \frac{nK^3}{2}
\]
faces.
The inequality follows from $kn \leq Kn$ and from the fact that there are at most $\binom{K}{2}<K^2/2$ pairs $(i,k)$ with $1 \leq i < k \leq K$.

Since the sets $\mathcal{F}_{i,k}$ are pairwise disjoint, the complexity of $\mathcal{F}$ is at least 
\[
\sum_{k=2}^K\sum_{\substack{i=1\\ \gcd(i,k)=1}}^{k-1} (n^2-2kn) =
\sum_{k=2}^K \varphi(k-1)(n^2-2kn) >
n^2\sum_{j=1}^{K-1}\varphi(j)-nK^3,
\]
where $\varphi(j)$ denotes the Euler's totient function.
The totient summatory function satisfies $\sum_{j=1}^{m}\varphi(j)\geq \frac{3m^2}{\pi^2}-O(m\log{m})$~\cite[pages 268--269]{Hardy79}.
Thus the complexity of $\mathcal{F}$ is at least $\frac{3n^2K^2}{\pi^2}-nK^3-O(n^2K\log{K})$.

Let $M$ be a given integer that satisfies $8n \leq M \leq n^4/8$.
We set $K \colonequals (M/n)^{1/3}$. 
We may assume that $K$ is an integer, as it does not affect the asymptotics.
For $8n \leq M \leq n^4/8$, we have $2 \leq K \leq n/2$.
The set $\mathcal{F}$ then contains at most $M$ faces and its complexity is at least 
\[\frac{3}{\pi^2}n^{4/3}M^{2/3}-M-O(M^{1/3}n^{5/3}\log{(M/n)}),\]
which is $\Omega(n^{4/3}M^{2/3})$ for a sufficiently large absolute constant $C$ and $Cn \leq M \leq n^4/C$.

\begin{remark}
We let $k,M \in \mathbb{N}$, $n=k^2$ and assume that $\Omega(\sqrt{n}) \leq M \leq O(n^2)$.
If we expand the edges of a regular bipartite drawing $K_{k,k}$ to lines, we obtain a set of $n$ lines and $M$ points 
with $\Omega((Mn)^{2/3})$ incidences, which is the maximum possible number, up to a multiplicative constant.
This construction is, however, only a transformation of the construction of Erd\H{o}s described for example 
in a paper of Pach and T\'{o}th~\cite{PachToth}.

\emph{Geometric duality} is a mapping from points and lines in the plane to lines and points, respectively, 
where the lines are restricted to those with finite slope.
A point $(a,b) \in \mathbb{R}^2$ is mapped to the line whose every point $(x,y)$ satisfies $y = ax-b$.
Notice that the duality preserves point-line incidences.
Let $l_{i,j}$ be the line passing through the points $(0,i)$ and $(1,j)$.
Its dual is the point $(j-i, -i)$.
Thus, the set of lines $l_{i,j}$ with $i,j \in \{0, \ldots, k-1\}$ is mapped to the set of points 
with integer coordinates inside the quadrangle with vertices $(k-1, 0)$, $(0,0)$, $(-(k-1), -(k-1))$ and $(0, -(k-1))$.
This point set can be transformed by a linear transformation of the plane into the set $\{0, 1, \ldots, k-1\}^2$,
which is the point set used in the construction of Erd\H{o}s.
\end{remark}

\end{document}